\documentclass[reqno]{amsart}
\usepackage{hyperref}
\usepackage{amssymb}

\allowdisplaybreaks

\theoremstyle{plain}
\newtheorem{theorem}{Theorem}[section]

\newtheorem{lemma}[theorem]{Lemma}

\theoremstyle{definition}
\newtheorem{definition}[theorem]{Definition}

\theoremstyle{remark}
\newtheorem{remark}[theorem]{Remark}
\numberwithin{equation}{section}

\begin{document}

\title[Blow-up of solutions]{Blow-up of solutions for a semilinear parabolic 
equation with nonlinear memory and absorption under  nonlinear nonlocal boundary condition}

\author[A. Gladkov]{Alexander Gladkov}
\address{Alexander Gladkov \\ Department of Mechanics and Mathematics
\\ Belarusian State University \\  4  Nezavisimosti Avenue \\ 220030
Minsk, Belarus  }    \email{gladkoval@bsu.by}

\subjclass[2020]{ 35K20, 35K58, 35K61}
\keywords{Parabolic equation, nonlinear memory, nonlocal boundary condition, blow-up, global existence}

\begin{abstract}
In this paper we consider initial boundary value problem for a parabolic equation 
with nonlinear memory and absorption under nonlinear nonlocal boundary condition.
We prove global existence and blow-up of solutions.

\end{abstract}

\maketitle

\section{Introduction}\label{sec1}

In this paper we consider the following parabolic equation 
with nonlinear memory and absorption 
\begin{equation}
    u_t= \Delta u + a \int_0^t u^q (x,\tau) \, d\tau - b u^m,\;x\in\Omega,\;t>0, \label{v:u}
\end{equation}
under nonlinear nonlocal boundary condition
\begin{equation}
\frac{\partial u(x,t)}{\partial\nu}=
\int_{\Omega}{k(x,y,t)u^l(y,t)}\,dy, \; x\in\partial\Omega, \; t > 0, \label{v:g}
\end{equation}
and initial datum
\begin{equation}
    u(x,0)=u_{0}(x),\; x\in\Omega, \label{v:n}
\end{equation}
where $ a, b,\,q, \,m,\,l $ are positive numbers, $\Omega$ is a bounded domain in $\mathbb{R}^N$
for $N\geq1$ with smooth boundary $\partial\Omega$, $\nu$ is unit
outward normal on $\partial\Omega.$

Throughout this paper we suppose that the functions
$k(x,y,t)$ and $u_0(x)$ satisfy the following conditions:
\begin{equation*}
k(x, y, t)\in
C(\partial\Omega\times\overline{\Omega}\times[0,+\infty)),\;k(x,y,t) \geq 0;
\end{equation*}
\begin{equation*}
u_0(x)\in C^1(\overline{\Omega}),\;u_0(x)\geq0\textrm{ in
}\Omega,\;\frac{\partial u_0(x)}{\partial\nu}=\int_{\Omega}{k(x,
y,0)u_0^l(y)}\,dy\textrm{ on }\partial\Omega.
\end{equation*}
Various phenomena in the natural sciences and engineering lead to the nonclassical mathematical models subject to nonlocal boundary conditions.
For global existence and blow-up of solutions for parabolic equations and systems
with nonlocal boundary conditions we refer to \cite{F} --  \cite{KD} and the references therein.
In particular, the blow-up problem for parabolic equations with nonlocal boundary condition
\begin{equation}\label{D}
    u(x,t)=\int_{\Omega}k(x,y,t)u^l(y,t)\,dy,\;x\in\partial\Omega,\;t>0,
\end{equation}
was considered in~\cite{GG} --  \cite{GK4}. Blow-up problem for (\ref{v:u}), (\ref{D}), (\ref{v:n}) was studied in~\cite{LMA}. 
Initial boundary value problems for parabolic equations with
nonlocal boundary condition (\ref{v:g}) were studied in~\cite{GK1} --  \cite{G3}.
So, the problem~(\ref{v:u})--(\ref{v:n}) with $a = 0$ was investigated in~\cite{G2,G1}.
Initial-boundary value problems with nonlinear memory for parabolic equations
 were addressed in many papers also (see, for example, \cite{LX} --  \cite{BSU}). 
 Local existence of solutions and comparison principle for~(\ref{v:u})--(\ref{v:n}) have been proved in \cite{BSU}. 

The aim of this paper is to investigate global existence and blow-up of solutions of~(\ref{v:u})--(\ref{v:n}).

This paper is organized as follows. In the next section we present finite time blow-up results for solutions with nontrivial initial data 
and with large initial data. In Section 3 we prove the global existence of solutions for any initial data.


\section{Blow-up in finite time}\label{blow}

 We begin with definition of a supersolution, a subsolution and a solution of~(\ref{v:u})--(\ref{v:n}). Let
$Q_T=\Omega\times(0,T),\;S_T=\partial\Omega\times(0,T)$,
$\Gamma_T=S_T\cup\overline\Omega\times\{0\}$, $T>0$.
\begin{definition}\label{v:sup}
We say that a nonnegative function $u(x,t)\in C^{2,1}(Q_T)\cap
C^{1,0}(Q_T\cup\Gamma_T)$ is a supersolution
of~(\ref{v:u})--(\ref{v:n}) in $Q_{T}$ if
        \begin{equation}\label{v:sup^u}
u_t \geq \Delta u + a \int_0^t u^q (x,\tau) \, d\tau - b u^m,\;(x,t)\in Q_T,
        \end{equation}
        \begin{equation}\label{v:sup^g}
\frac{\partial u(x,t)}{\partial\nu}\geq\int_{\Omega}{k(x, y,
t)u^l(y, t) }\,dy, \; x \in \partial \Omega,\; 0 < t < T,
        \end{equation}
        \begin{equation}\label{v:sup^n}
            u(x,0)\geq u_{0}(x),\; x\in\Omega,
        \end{equation}
and $u(x,t)\in C^{2,1}(Q_T)\cap C^{1,0}(Q_T\cup\Gamma_T)$ is a
subsolution of~(\ref{v:u})--(\ref{v:n}) in $Q_{T}$ if $u\geq0$ and
it satisfies~(\ref{v:sup^u})--(\ref{v:sup^n}) in the reverse
order. We say that $u(x,t)$ is a solution of
problem~(\ref{v:u})--(\ref{v:n}) in $Q_T$ if $u(x,t)$ is both a
subsolution and a supersolution of~(\ref{v:u})--(\ref{v:n}) in
$Q_{T}$.
\end{definition}

To prove the main results we use the comparison principle and the positiveness of a solution
which have been proved in~\cite{BSU}.

\begin{theorem}\label{Th1} Let $\overline{u}$ and $\underline{u}$ be a
 supersolution and a  subsolution of problem
(\ref{v:u})--(\ref{v:n}) in $Q_T,$ respectively. Suppose that
$\underline{u}(x,t)> 0$ or $\overline{u}(x,t) > 0$ in ${Q}_T\cup
\Gamma_T$ if $\min (q, l) < 1.$  Then $ \overline{u}(x,t) \geq
\underline{u}(x,t) $ in ${Q}_T\cup \Gamma_T.$
\end{theorem}
\begin{theorem}\label{Th2} Let ${u}$ be a solution of problem
(\ref{v:u})--(\ref{v:n}) in $Q_T.$  Let $ u_0(x) \not\equiv 0$ in $\Omega$ and $m \ge 1.$ Then 
$u(x,t)> 0$ in ${Q}_T\cup S_T.$ If $ u_0(x) > 0$ in $\overline\Omega$ and  $m < 1$  then $u(x,t)> 0$ in ${Q}_T\cup \Gamma_T.$
\end{theorem}

To prove blow-up result we need the following statement.
\begin{lemma}\label{positive}
Let ${u}$ be a solution of problem (\ref{v:u})--(\ref{v:n}) in $Q_T$  with $u_0(x) \not\equiv 0.$ Then 
\begin{equation}\label{L}
\int_{ \Omega} u(x,t) \, dx \ge c_1  \; \textrm{ for} \; t \in [0, T],
\end{equation}
where positive constant $c_1$ does not depend on $T.$
\end{lemma}
\begin{proof}
Since $u_0(x) \not\equiv 0$ there exists $t_0 > 0$ such that
\begin{equation}\label{L1}
 \int_{\Omega} u(x,t) \, dx > 0   \; \textrm{ for } \; t \in [0, t_0].
\end{equation}
Suppose that $T>t_0$ and $ m \ge 1.$ By Theorem~\ref{Th2} we have
\begin{equation*}\label{L2}
u(x,t) > 0  \; \textrm{ in} \; {Q}_T\cup S_T.
\end{equation*}
Then there exists $c_2 > 0$ such that
\begin{equation*}\label{L3}
u(x,t_0) \ge c_2  \; \textrm{and} \; a \int_{0}^{t_0} u^q(x,t) \, dt \ge c_2   \; \textrm{ for } \; x \in \overline{\Omega}.
\end{equation*}

Let us consider the following problem
 \begin{equation}\label{L4}
v_t = \Delta v + a \int_0^{t_0} u^q (x,\tau) \, d\tau + a \int_{t_0}^t v^q (x,\tau) \, d\tau  - b v^m,\;(x,t) \in Q_T \setminus \overline{Q_{t_0}},
        \end{equation}
        \begin{equation}\label{L5}
\frac{\partial v(x,t)}{\partial\nu} = \int_{\Omega}{k(x, y,
t) v^l(y, t) }\,dy, \; x \in \partial \Omega,\; t_0 < t < T,
        \end{equation}
        \begin{equation}\label{L6}
            v(x,t_0) = u(x,t_0),\; x\in\Omega.
        \end{equation}
Let
\begin{equation*}
c_3 = \min \left\{c_2, \left[ \frac{c_2}{b} \right]^\frac{1}{m} \right\}.
\end{equation*}        
It is easy to see that $u(x,t)$ and $v(x,t) = c_3$ are a solution and a subsolution of problem (\ref{L4})--(\ref{L6}) in $Q_T \setminus \overline{Q_{t_0}},$ respectively. By comparison principle for (\ref{L4})--(\ref{L6}) (see Theorem~3 in~\cite{BSU} for similar problem)
\begin{equation}\label{L7}
u(x,t) \ge c_3  \; \textrm{in} \; Q_T \setminus \overline{Q_{t_0}}.
\end{equation}
From (\ref{L1}), (\ref{L7}) we obtain (\ref{L}).

Now let  $ m < 1.$ 
Integrating (\ref{v:u}) over $\Omega$ and using Green's identity, we obtain 
\begin{equation}\label{L8}
\begin{split}
 \int_\Omega  u_t (x,t) \, dx & =  \int_{\partial \Omega} \int_\Omega k(x,y,t) u^l(y,t) \,dy \,dS_x + a \int_{\Omega} \int_0^t  u^q (x,\tau) \, d\tau \, dx  \\ & -  b  \int_\Omega u^m (x,t) \, dx. 
\end{split}
\end{equation}
Let us introduce auxiliary function
\begin{equation*}
J_1 (t) =  \int_{\Omega} u(x,t) \, dx.
\end{equation*}
Applying to (\ref{L8}) H\"older's inequality, we have
\begin{equation*}
J'_1 (t)  > c_4 - b \vert\Omega\vert^{1-m} J^m_1 (t)  \; \textrm{for} \; t \ge t_0,
\end{equation*}
where
\begin{equation*}
c_4 =  \frac{a}{2} \int_0^{t_0} \int_{\Omega} u^q (x,\tau) \, dx \, d\tau > 0.
\end{equation*}
Let
\begin{equation*}
c_5 = \min \left\{ J_1 (t_0), \left[ \frac{c_4}{b \vert\Omega\vert^{1-m}} \right]^\frac{1}{m}\right\}.
\end{equation*}
Obviously, $J'_1 (t) > 0$ if  $J_1 (t) \le \left[ \frac{c_4}{b \vert\Omega\vert^{1-m}} \right]^{1/m}.$ Therefore
\begin{equation*}
J_1 (t)  \ge c_5  \; \textrm{for} \; t \ge t_0
\end{equation*}
and  (\ref{L})  holds.
\end{proof}


To formulate finite time blow-up result we introduce
\begin{equation*}
\underline k(t) = \inf_{\Omega} \int_{\partial \Omega} k(x,y,t) \, dS_x
\end{equation*}
 and suppose that
\begin{equation}\label{E9}
\underline k(0) > 0
\end{equation}
and for some  $k_1 >0, \,$ $t_1 >0$
\begin{equation}\label{E91}
\int_{\partial \Omega} k(x,y,t) \, dS_x \ge k_1  \; \textrm{ for any} \;  y \;  \in \Omega  \; \textrm{ and} \;t \ge t_1.
\end{equation}

\begin{theorem}\label{blow-up}
Let $q>\max(m,1)$ or $l>\max(m,1)$ and (\ref{E91}) hold. Then solutions of~(\ref{v:u})--(\ref{v:n}) blow up in finite time for any $u_0(x) \not\equiv 0.$ If $l>\max(m,1)$ and (\ref{E9}) holds then solutions of~(\ref{v:u})--(\ref{v:n}) blow up in finite time for large enough initial data.
\end{theorem}
\begin{proof}
Suppose that $q>\max(m,1)$ and $u_0(x) \not\equiv 0.$ Let us introduce auxiliary function
\begin{equation*}
J_2 (t) =  \int_0^t \int_{\Omega} u^q (x,\tau) \,  dx d\tau.
\end{equation*}
Applying H\"older's inequality to the right hand side of (\ref{L8}), we obtain 
\begin{equation}\label{E101}
J_1' (t) \geq a J_2 (t) - b \vert\Omega\vert^{(q-m)/q} \left( J_2' (t) \right)^{m/q}.
\end{equation}
By H\"older's inequality, we have
\begin{equation}\label{E111}
J_2' (t) =  \int_{\Omega} u^q (x,t) \, dx \geq  \vert\Omega\vert^{1-q} J_1^q (t).
\end{equation}
Slightly modifying the proof of Lemma 5.3 in \cite{Souplet} from (\ref{L}), (\ref{E101}), (\ref{E111}) we conclude that there exists $T \in (0, \infty)$ such that
\begin{equation*}
\lim_{t \to T} (J_1 (t) + J_2 (t) ) = \infty,
\end{equation*}
which implies that $u(x,t)$ blows up in a finite time for any nonnegative nontrivial initial datum.

Let  $l>\max(m,1)$ and (\ref{E91}) hold now. Applying H\"older's inequality for $t \ge t_1$ to the right hand side of (\ref{L8}), we obtain 
\begin{equation}\label{E10}
J_1' (t) \geq  a \int_0^t \int_{\Omega}  u^q (x,\tau) \,  dx d\tau + k_1 J_3 (t) - b \vert\Omega\vert^{(l-m)/l}  J_3^{m/l},
\end{equation}
where
\begin{equation*}
J_3 (t) =  \int_{\Omega} u^l (x,t) \,  dx.
\end{equation*}
We note that 
\begin{equation}\label{E10a}
a  \int_0^t \int_{\Omega} u^q (x,\tau) \,  dx d\tau + k_1 J_3 (t) - b \vert\Omega\vert^{(l-m)/l}  J_3^{m/l} \ge \frac{k_1}{2} J_3 (t)
\end{equation}
if
\begin{equation}\label{E10b}
a \int_0^t \int_{\Omega}  u^q (x,\tau) \,  dx d\tau  \ge b^{l/(l-m)} \vert\Omega\vert \left[  \frac{2 }{ k_1 } \right]^\frac{m}{l-m}.
\end{equation}
Then from (\ref{E10}) -- (\ref{E10b}) using H\"older's inequality, we obtain for large values of $t$ 
\begin{equation}\label{E11}
J_1' (t) \geq \frac{k_1}{2}  J_3 (t) \geq \frac{k_1}{2} \vert\Omega\vert^{1-l} J_1^l (t).
\end{equation}
Thus solutions of~(\ref{v:u})--(\ref{v:n}) blow up in finite time for any $u_0(x) \not\equiv 0$ if
\begin{equation*}
a \int_0^\infty \int_{\Omega} u^q (x,\tau) \,  dx d\tau = \infty.
\end{equation*}

We suppose now that 
\begin{equation}\label{E12}
\int_0^\infty \int_{\Omega} u^q (x,\tau) \,  dx d\tau = M < \infty.
\end{equation}
Integrating (\ref{L8}) by $t$, we obtain 
\begin{equation}\label{E13}
\begin{split}
 \int_\Omega  u (x,t) \, dx &=  \int_\Omega  u_0 (x) \, dx + \int_{0}^{t} \int_{\partial \Omega} \int_\Omega k(x,y,\tau) u^l(y,\tau) \,dy \,dS_x \, d\tau \\ & 
+ a \int_{0}^{t} \int_{\Omega} \int_0^\tau  u^q (x,\sigma) \, d\sigma \, dx \, d\tau -  b \int_{0}^{t} \int_\Omega u^m (x,\tau) \, dx \, d\tau. 
\end{split}
\end{equation}
Using H\"older's inequality and Lemma~\ref{positive}, we have
\begin{equation}\label{E14}
\int_{t_1}^{t} \int_\Omega u^l(y,\tau) \,dy \, d\tau \ge [(t - t_1) \vert\Omega\vert]^{-(l-1)} \left[ \int_{t_1}^{t} \int_\Omega u(y,\tau) \,dy \, d\tau \right]^l \ge c_1^l \vert\Omega\vert^{-(l-1)}(t - t_1).
\end{equation}

If $m \leq q$ then applying H\"older's inequality, we get
\begin{equation}\label{E15}
\int_{0}^{t} \int_\Omega u^m(y,\tau) \,dy \, d\tau \le (t \vert\Omega\vert)^{(q-m)/q}  \left\{ \int_{0}^{t} \int_\Omega u^q(y,\tau) \,dy \, d\tau \right\}^\frac{m}{q} \le(t \vert\Omega\vert)^{(q-m)/q}  M^{m/q} .
\end{equation}
Now from (\ref{E91}), (\ref{E13}) -- (\ref{E15}) we conclude that $J_1 (t) \to \infty$ as $t \to \infty.$ Then by H\"older's inequality $J_3 (t) \to \infty$ as $t \to \infty$ and, moreover, (\ref{E10a}), (\ref{E11}) hold for large values of $t.$ Thus solutions of~(\ref{v:u})--(\ref{v:n}) blow up in finite time for any $u_0(x) \not\equiv 0.$

If $m >q$ then it is easy to check that 
\begin{equation}\label{E16}
b u^m (x,t) \le b \left[ \frac{2b}{k_1} \right]^\frac{m-q}{l-m} u^q (x,t) + \frac{k_1}{2} u^l (x,t).  
\end{equation}
Now from (\ref{E91}), (\ref{E12}) -- (\ref{E14}), (\ref{E16}) we conclude $J_1 (t) \to \infty$ as $t \to \infty$ that guarantees blow up in finite time  of solutions of~(\ref{v:u})--(\ref{v:n}) for any $u_0(x) \not\equiv 0.$


Let  $l>\max(m,1)$ and (\ref{E9}) hold now. Then there exists $T_0 >0$ such that $\underline k(t) > 0$ for $ t \in [0, T_0].$ Denote
$$
k_0 = \min_{[0, T_0]} \underline k(t).
$$
Applying H\"older's inequality to (\ref{L8}), we obtain for $t \leq T_0$
\begin{equation}\label{10}
J_1' (t) \geq   k_0 J_3 (t) -  b  \vert\Omega\vert^{(l-m)/l} J_3^{m/l} (t) = J_3 (t) \left[  k_0  -  b  \vert\Omega\vert^{(l-m)/l} J_3^{(m-l)/l} (t) \right]. 
\end{equation}
By H\"older's inequality, we have
\begin{equation}\label{11}
J_3 (t) =  \int_{\Omega} u^l (x,t) \, dx \geq  \vert\Omega\vert^{1-l} J_1^l (t).
\end{equation}
Let
\begin{equation}\label{111}
J_1 (0) \geq \vert\Omega\vert \left\{ \frac{2b}{k_0} \right\}^\frac{1}{l-m}.
\end{equation}
Then
\begin{equation*}
J_3 (0) \geq \vert\Omega\vert \left\{ \frac{2b}{k_0} \right\}^\frac{l}{l-m},
\end{equation*}
and, moreover, 
\begin{equation}\label{12}
k_0  -  b  \vert\Omega\vert^{(l-m)/l} J_3^{(m-l)/l} (t) \geq \frac{k_0 }{2}
\end{equation}
for any $t>0.$ From (\ref{10}) -- (\ref{12})  we obtain
\begin{equation*}
J_1' (t) \geq \frac{k_0  \vert\Omega\vert^{1-l}}{2 } J_1^l (t). 
\end{equation*}
Obviously, $J_1 (t)$ blows up at $t \leq T_0$ if $J_1 (0)$ satisfies (\ref{111}) and 
\begin{equation*}
J_1 (0) \geq \left[ (l-1) \frac{k_0  \vert\Omega\vert^{1-l}}{2 } T_0  \right]^{-\frac{1}{l-1}}.
\end{equation*}
\end{proof}


\section{Global existence}\label{gl}
The proof of a global existence result relies on the continuation principle and the
construction of a supersolution. 

\begin{theorem}\label{global}
Let at least one from the following conditions hold:

a). $\max (q,l) \leq 1;$ 

b). $\max (q,l) > 1$ and $ l < m, \, q \leq m.$

Then every solution of (\ref{v:u})--(\ref{v:n}) is global.
\end{theorem}
\begin{proof}
In order to prove global existence of solutions we construct a
suitable explicit supersolution of~(\ref{v:u})--(\ref{v:n}) in
$Q_T$ for any positive $T.$ Suppose at first that  $\max (q,l) \leq 1.$ 
Since $k(x,y,t)$ is a continuous function, there exists a constant
$K>0$ such that
\begin{equation*}\label{K}
k(x,y,t)\leq K
\end{equation*}
in $\partial\Omega\times Q_T.$ Let $\lambda_1$ be the first
eigenvalue of the following problem
\begin{equation*}\label{EF}
    \begin{cases}
        \Delta\varphi+\lambda\varphi=0,\;x\in\Omega,\\
        \varphi(x)=0,\;x\in\partial\Omega,
    \end{cases}
\end{equation*}
and $\varphi(x)$ be the corresponding eigenfunction with
$\sup\limits_{\Omega}\varphi(x)=1$. It is well known,
$\varphi(x)>0$ in $\Omega$ and $\max\limits_{\partial\Omega}
\partial\varphi(x)/\partial\nu < 0.$

Now, let $\overline u (x,t) $ be defined as 
\begin{equation}\label{U}
\overline u (x,t) = \frac{C\exp (\mu t)}{c \varphi (x) + 1},
\end{equation}
where constants $C,\mu$ and $c$ are chosen to satisfy the
inequalities:
\begin{equation}\label{c}
c\geq \max \left\{ K \int_\Omega \frac{dy}{(\varphi (y) + 1)^l}
\max_{\partial \Omega} \left(-\frac{\partial \varphi}{\partial
\nu} \right)^{-1}, 1 \right\},
\end{equation}
\begin{equation}\label{C}
C \geq \max \{ \sup_\Omega (c \varphi (x) + 1) u_0 (x), 1 \}  
\end{equation}
and 
\begin{equation*}
\mu \geq \lambda_1 + 2 c^2 \sup_\Omega \frac{\vert\nabla \varphi
\vert^2 } {(c \varphi (x) + 1)^2} + a (c+1)^{1-q} + \frac{1}{q}.
\end{equation*}
It is not difficult to check that $\overline u (x,t) $ is a supersolution of~(\ref{v:u})--(\ref{v:n}) in $Q_T.$ Then by Theorem~\ref{Th1} 	
\begin{equation*}
  u(x,t) \leq \overline u (x,t) \,\,\,  \textrm{in} \,\,\,
 \overline{Q}_T.
	\end{equation*}

For $l \leq 1, \, 1 < q \leq m $ the function $\overline u (x,t)$ in  (\ref{U}) is a supersolution of~(\ref{v:u})--(\ref{v:n}) in $Q_T$ if (\ref{c}), (\ref{C}) hold and 
\begin{equation*}
\mu \geq \lambda_1 + 2 c^2 \sup_\Omega \frac{\vert\nabla \varphi
\vert^2 } {(c \varphi (x) + 1)^2} + \frac{ a (c+1)^{m-q}}{qb}.
\end{equation*}

Suppose now that $1 < l < m, \, q \leq m.$ To construct a supersolution we use the change of variables in a
	neighborhood of $\partial \Omega$ as in \cite{CPE}. Let
	$\overline x\in\partial \Omega$ and $\widehat{n}
	(\overline x)$ be the inner unit normal to $\partial \Omega$ at the
	point $\overline x.$ Since $\partial \Omega$ is smooth it is well
	known that there exists $\delta >0$ such that the mapping $\psi
	:\partial \Omega \times [0,\delta] \to \mathbb{R}^n$ given by
	$\psi (\overline x,s)=\overline x +s\widehat{n} (\overline x)$
	defines new coordinates ($\overline x,s)$ in a neighborhood of
	$\partial \Omega$ in $\overline\Omega.$ A straightforward
	computation shows that, in these coordinates, $\Delta$ applied to
	a function $g(\overline x,s)=g(s),$ which is independent of the
	variable $\overline x,$ evaluated at a point $(\overline x,s)$ is
	given by
	\begin{equation}\label{Gl:new-coord}
	\Delta g(\overline x,s)=\frac{\partial^2g}{\partial s^2}(\overline x,s)-\sum_{j=1}^{n-1}\frac{H_j(\overline x)}{1-s
		H_j (\overline x)}\frac{\partial g}{\partial s}(\overline x,s),
	\end{equation}
	where $H_j (\overline x)$ for $j=1,...,n-1,$ denote the principal
	curvatures of $\partial\Omega$ at $\overline x.$ For $0\leq s\leq \delta$
	and small $\delta$  we have
	\begin{equation}\label{Gl:enq4}
	\left\vert\sum_{j=1}^{n-1} \frac{H_j (\overline x)}{1-s H_j (\overline
		x)}\right\vert\leq\overline c.
	\end{equation}
	
	Let $0<\varepsilon<\omega<\min(\delta, 1),\, $ $r > a/(bq), \,$
	$\max(1/l, 2/(m-1)) <\beta< 2/(l-1),\,$  $0<\gamma<\beta/2,\,$ $A\ge\sup_\Omega u_0(x).$
	We modify a supersolution in \cite{GG1}. For points in $Q_{\delta,T}=\partial \Omega \times [0,
	\delta]\times [0,T]$ of coordinates $(\overline x,s,t)$ define
	\begin{equation}\label{Gl:function}
	\overline v (x,t)= \overline v ((\overline x,s),t)= \left(  \left[(s+\varepsilon)^{-\gamma}-\omega^{-\gamma}\right]_+^\frac{\beta}{\gamma} + A \right) \exp (rt),
	\end{equation}
	where $s_+=\max(s,0).$ For points in $\overline{Q_T}\setminus Q_{\delta,T}$
	we set  $ \overline v(x,t)= A \exp (rt).$ We prove that $ \overline v(x,t)$
	is the supersolution of~(\ref{v:u})--(\ref{v:n}) in $Q_T.$
	It is not difficult to check that
	\begin{equation}\label{Gl:enq1}
	\left\vert\frac{\partial \overline v}{\partial s}\right\vert\leq \beta\min\left(\left[ D(s)\right]^\frac{\gamma+1}{\gamma}\left[( s+\varepsilon)^{-\gamma}-\omega^{-\gamma}\right]_+^\frac{\beta+1}{\gamma},\,(s+\varepsilon)^{-(\beta+1)}\right) \exp (rt),
	\end{equation}
	\begin{equation}\label{Gl:enq2}
	\left\vert\frac{\partial^2 \overline v}{\partial s^2}\right\vert\leq \beta (\beta+1)\min\left(\left[D(s)\right]^\frac{2(\gamma+1)}{\gamma}\left[( s+\varepsilon)^{-\gamma}-
	\omega^{-\gamma}\right]_+^\frac{\beta+2}{\gamma},\,( s+\varepsilon)^{-(\beta+2)}\right) \exp (rt),
	\end{equation}
	where
	\begin{equation*}
	D(s)= \frac{( s+\varepsilon)^{-\gamma}}{ (s+\varepsilon)^{-\gamma}-\omega^{-\gamma}}.
	\end{equation*}
	Then $D^\prime(s)>0$ and for any $\overline\varepsilon>0$
	\begin{equation}\label{Gl:enq3}
	1\leq D(s)\leq 1+\overline\varepsilon, \; 0<s\leq{\overline s},
	\end{equation}
	where ${\overline s} = [\overline\varepsilon/(1+\overline\varepsilon)]^{1/\gamma}\omega-\varepsilon,$
	$\varepsilon<[\overline\varepsilon/(1+\overline\varepsilon)]^{1/\gamma}\omega.$ We denote
\begin{equation}\label{Gl}
	Lv \equiv v_t - \Delta v - a \int_0^t v^q (x,\tau) \, d\tau + b v^m.
	\end{equation}
By (\ref{Gl:new-coord})--(\ref{Gl})  we can choose
	$\overline\varepsilon$ small and  $A$ large so that in $Q_{{\overline s},T}$
\begin{align*}
        L\overline v &\geq \left\{ r \left(  \left[(s+\varepsilon)^{-\gamma}-\omega^{-\gamma}\right]_+^\frac{\beta}{\gamma} + A \right) +  b \left( \left[( s+\varepsilon)^{-\gamma}-\omega^{-\gamma}\right]_+^\frac{\beta }{\gamma} + A \right)^m  \exp [r(m-1)t] \right. \\
         &-  \beta(\beta+1)\left[ D(s)\right]^\frac{2(\gamma+1)}{\gamma}\left[( s+\varepsilon)^{-\gamma}-\omega^{-\gamma}\right]_+^\frac{\beta+2}{\gamma} +  \beta \overline c \left[ D(s)\right]^\frac{\gamma+1}{\gamma} \left[( s+\varepsilon)^{-\gamma}-\omega^{-\gamma}\right]_+^\frac{\beta+1}{\gamma}  \\
       &- \left. \frac{a}{rq} \left(\left[( s+\varepsilon)^{-\gamma}-\omega^{-\gamma}\right]_+^\frac{\beta}{\gamma}+A\right)^q \exp [r(q-1)t] \right\} \exp (rt) \geq 0.
      \end{align*}
	
	Let $s\in[{\overline s},\delta].$ From (\ref{Gl:new-coord})--(\ref{Gl:enq2}) we have
	\begin{equation*}
	\vert\Delta \overline v\vert \leq \left\{  \beta(\beta+1)\omega^{-(\beta+2)}\left(\frac{1+\overline\varepsilon}
	{\overline\varepsilon}\right)^\frac{\beta+2}{\gamma} + \beta\overline c\omega^{-(\beta+ 1)}
\left(\frac{1+\overline\varepsilon}{\overline\varepsilon}\right)^\frac{\beta+1}{\gamma} \right\} \exp (rt)
	\end{equation*}
	and  $L\overline v\geq0$ for large values of $A.$ Obviously, in $\overline{Q_T}\setminus Q_{\delta,T}$
	\begin{equation*}
	L\overline v = rA \exp (rt) - \frac{aA^q}{rq} [\exp (rqt) - 1]  + b A^m \exp (rmt) \geq  0
	\end{equation*}
	for  $A \ge 1.$
	
	Now we prove the following inequality
\begin{equation}\label{E:4.6}
\frac{\partial \overline v}{\partial\nu} (\overline x,0,t) \geq
\int_{\Omega} K \overline v^l(\overline x,s,t) \, dy, \quad
(x,t) \in S_T
\end{equation}
for a suitable choice of $\varepsilon.$ To estimate the integral
$I$ in the right hand side of (\ref{E:4.6}) we use the change of
variables in a neighborhood of $\partial\Omega$ as above. Let
	\begin{equation*}
 \overline J= \sup_{0< s< \delta} \int_{\partial\Omega}
\vert J(\overline y,s)\vert \, d\overline y,
	\end{equation*}
where $J(\overline y,s)$ is Jacobian of the change of variables.
Then we have
\begin{align*}
I \leq &  2^{l-1}  K \left\{ \int_{\Omega} \left[ ( s +
\varepsilon)^{-\gamma} - \omega^{-\gamma}\right]_+^\frac{\beta
l}{\gamma} \, dy + A^l \vert\Omega\vert \right\} \exp (rlt) \\
\leq &  2^{l-1}  K  \left\{  \overline J \int_{0}^{\omega -
\varepsilon} \left[ ( s + \varepsilon)^{-\gamma} -
\omega^{-\gamma}\right]^\frac{\beta
l}{\gamma} \, ds + A^l \vert\Omega\vert \right\} \exp (rlt) \\
\leq &  2^{l-1}  K  \left\{ \frac{  \overline J}{\beta l-1} \left[
\varepsilon^{-(\beta l-1)} - \omega^{-(\beta l-1)}\right] +
 A^l \vert\Omega\vert \right\} \exp (rlt).
 \end{align*}
  On the other hand, since
	\begin{equation*}
\frac{\partial \overline v}{\partial\nu} (\overline x,0,t) = -
\frac{\partial \overline v}{\partial s} (\overline x,0,t) =
\beta \varepsilon^{-\gamma -1} \left[ \varepsilon^{-\gamma} -
\omega^{-\gamma}\right]^\frac{\beta-\gamma}{\gamma} \exp (rt),
	\end{equation*}
the inequality (\ref{E:4.6}) holds if $\varepsilon$ is small
enough. At last,
	\begin{equation*}
 u(x,0) \leq \overline v (x,0) \,\,\,  \textrm{in} \,\,\,  \Omega.
	\end{equation*}
Hence, by Theorem~\ref{Th1} we get
	\begin{equation*}
  u(x,t) \leq \overline v (x,t) \,\,\,  \textrm{in} \,\,\,
 \overline{Q}_T.
	\end{equation*}
\end{proof}

\begin{remark}
In the case $q \leq m, \, l > 1, \, l = m$ global existence and blow-up results depend on $b$ and $k(x,y,t).$ 
 \end{remark}

\subsection*{Acknowledgements}
This work is supported by the state program of fundamental research of Belarus
(grant 1.2.03.1).


\end{document}